\documentclass{amsart}

\usepackage{amsmath} % Pour utiliser la commande \boldsymbol qui permet de mettre des lettres grecs en gras
\usepackage{amsfonts} % Permet l'utilisation de plus de polices de caractères en mode mathématique
\usepackage{amssymb} % Pour faciliter l'écriture de théorèmes mathématiques
\usepackage{epsfig} % Pour utiliser la commande \epsfig
\usepackage{graphicx} % Pour utiliser la commande \includegraphics
\usepackage{multirow} % Pour faire des tableaux contenant des lignes fusionnées
\usepackage{verbatim} % Permet d'insérer un fichier ASCII brut
\usepackage{rotating} % Perment l'utilisation de l'environnement sideways
\usepackage[all]{xy}

\newtheorem{theorem}{Theorem}[section]
\newtheorem{lemma}[theorem]{Lemma}

\theoremstyle{definition}
\newtheorem{definition}[theorem]{Definition}

\theoremstyle{remark}
\newtheorem{remark}[theorem]{Remark}
\theoremstyle{notation}
\newtheorem{notation}[theorem]{Notation}
\numberwithin{equation}{section}
\theoremstyle{corollary}
\newtheorem{corollary}[theorem]{Corollary}

\usepackage[bookmarks=false]{hyperref}
\newcommand{\Map}{\mathbf{Map}}
\newcommand{\map}{\mathbf{map}}

\newcommand{\HOM}{\mathbf{HOM}}
\newcommand{\homs}{\mathbf{hom}}
\newcommand{\sSet}{\mathbf{sSet}}
\newcommand{\Spn}{\mathbf{Sp}^{\mathbb{N}}}

\newcommand{\Cat}{\mathbf{Cat}}

\newcommand{\C}{\mathbf{C}}

\newcommand{\D}{\mathbf{D}}
\newcommand{\A}{\mathbf{A}}
\newcommand{\B}{\mathbf{B}}
\newcommand{\M}{\mathbf{M}}

\newcommand{\sCat}{\mathbf{sCat}}

\newcommand{\K}{\mathbb{K}}

\newcommand{\N}{\mathrm{N}_{\bullet}}
\newcommand{\iso}{\mathrm{iso}}

\newcommand{\Ho}{\mathrm{Ho}}

\newcommand{\colim}{\mathrm{colim}}
\newcommand{\Ob}{\mathrm{Ob}}
\newcommand{\diag}{\mathrm{diag}}
\begin{document}

\title{Stabilization of the category of simplicial objects in CAT.}

\author{Amrani Ilias}
\address{Department of Mathematics\\ Masaryk University\\ Czech Republic}
%\curraddr{Department of Mathematics and Statistics,
%Case Western Reserve University, Cleveland, Ohio 43403}
\email{ilias.amranifedotov@gmail.com}
\email{amrani@math.muni.cz}
   % \thanks %will become a 1st page footnote.
\thanks{Supported by the project CZ.1.07/2.3.00/20.0003
of the Operational Programme Education for Competitiveness of the Ministry
of Education, Youth and Sports of the Czech Republic.
}

%\author{Author Two}
%\address{Mathematical Research Section, School of Mathematical Sciences,
%Australian National University, Canberra ACT 2601, Australia}
\thanks{}

%    General info
\subjclass[2000]{Primary 18, 55}

\date{May 25, 2012 }



\begin{abstract}
In this article, we define two equivalent new model structures on $ \sCat$ the category of simplicial objects in $\Cat$. Then we construct the corresponding stable model category of spectra $\Spn(\sCat)$ and make some links with the algebraic $\K$-theory via the mapping space. 
\end{abstract}

\maketitle
%------------------------------------------------------------------------------
% Introduction
%------------------------------------------------------------------------------

%------------------------------------------------------------------------------
% Part 1
%------------------------------------------------------------------------------
\section*{Introduction and main results}

We start by introducing two Quillen equivalent new model structure on $\sCat$ (i.e., $[\Delta^{op},\Cat]$), the category of simplicial objects in $\Cat$, we will often call them \textit{simplicial categories}. It is a \textbf{discrete} version of the diagonal model structure defined in [\cite{amrani2011grothendieck}, theorem 1.2], \textbf{but we will go further in our investigations, mainly the stabilization process}, as we will see later. Equipped with these new model structures, the mapping space $\map$ is closely related to the algebraic $\K$-theory. In order to get a full analogy, we construct the stable category of spectra $ \Spn(\sCat)$ following ideas of \cite{hovey2001}. Then the mapping space $\map_{\Spn(\sCat)}(\Sigma^{\infty}S^{0},\C_{\bullet}^{\bullet})$ is an infinity loop space, where $\C_{\bullet}^{\bullet}$ is a fibrant object in the model category $ \Spn(\sCat)$. Roughly speaking 
$\C_{\bullet}^{\bullet}$ is a \textit{categorical} $\Omega$-spectra. \\ 
We recall, that the $\K$-theory of an exact category or more generally a Waldhausen category with isomorphisms produces an $\Omega$-spectra. Let's $\C$ be a Waldhausen category, then the Waldhausen "suspension" 
$\mathcal{S}_{\bullet}\C$,  is a simplicial category and the $\K$-theory spectra is given by the following sequence \cite{waldhausen1983}:
$$\{\Omega\diag\N~\iso~ \mathcal{S}_{\bullet}\C, \diag\N~\iso~ \mathcal{S}_{\bullet}\C,\dots  ,\diag\N~\iso~ \mathcal{S}^{n}_{\bullet}\C,\dots\}$$
which is an $\Omega$-spectra i.e., a fibrant object in the stable model category $\Spn(\sSet_{\ast})$.\\
In section \ref{section1}, we construct the diagonal model structure on $\sCat$. We prove the following theorem:\\
\textbf{Theorem A: }(diagonal model structure \ref{structure})
\textit{The category of simplicial categories $\sCat$ is a cofibrantly generated model category where
}
\begin{enumerate}
\item \textit{a morphism $f:\C_{\bullet}\rightarrow\D_{\bullet}$ 
is a weak equivalence (resp. fibration) if and only if  $\diag\N\iso(f)$ is a weak equivalence (resp. fibration) in $\sSet$.}
\item \textit{The generating (acyclic)  cofibrations in $\sCat$ are given by the image of generating (acyclic) cofibration in $\sSet$ via the fonctor $\pi_{\bullet} d_{\ast}$. }
\end{enumerate}

In section \ref{section2}, we prove that the new model structure on $\sCat$ is cellular and proper.\\\\
\textbf{Theorem B: }(additional properties\ref{cellsacat},\ref{leftproper})
\textit{The diagonal model structure on $\sCat$ is left proper and cellular.}\\

 In section \ref{section3}, we equip the category of simplicial categories with a new model category equivalent to the diagonal one (same weak equivalences), \textbf{but having the advantage to be well tensored and cotensored with respect to the model category $\sSet$}. We establish the following theorem:\\

\textbf{Theorem C: }($\overline{W}$-model structure \ref{newstructure}) \textit{There exists a cofibrantly generated model structure on $\sCat$ induced by the adjunction }
$$\xymatrix{
\sSet \ar@<1ex>[r]^-{\pi_{\bullet}~Dec} & \sCat \ar@<1ex>[l]^-{\overline{W}\N\iso} 
}$$
\textit{Moreover, the $\overline{W}$-model structure on $\sCat$ is Quillen equivalent to the diagonal model structure, left proper and cellular.}\\

In section \ref{section4}, construct and explicit suspension functor and a loop functor in the pointed $\overline{W}$-model category $\sCat$.\\\\
\textbf{Theorem D:} (compatible (co)tensorization) \ref{S-Quillen}
\textit{
If $X_{\bullet}$ is a pointed simplicial set, then the functor
$$-\land X_{\bullet} :~ \sCat_{\ast}\rightarrow \sCat_{\ast}$$
 is a left Quillen functor, where $\sCat$ is equipped with $\overline{W}$-model structure. 
 Moreover, the functor $-\land X_{\bullet} $ has a right Quillen adjoint. 
}\\\\
In the final section \ref{section5}, we prove our main theorem, roughly speaking we construct a the stable model category of spectra $\Spn(\sCat_{\ast})$ and relate the mapping space $\map_{\Spn(\sCat_{\ast})}$
to the algebraic $\K$-theory:\\\\
\textbf{Theorem E: }(stabilization \ref{spectrificationscat})
\textit{
There is a cofibrantly generated stable model category structure on $\Spn(\sCat_{\ast},\Sigma)$.}\\\\

As consequence of the last theorem is that any fibrant object $\D^{\bullet}_{\bullet}$ in $\Spn(\sCat_{\ast},\Sigma)$ has the property that $\Omega\diag\N\iso \D_{\bullet}^{n+1}$ is equivalent to $\diag\N\iso \D_{\bullet}^{n}$, It means that $\D_{\bullet}^{n+1}$ looks like $\mathcal{S}_{\bullet}\D_{\bullet}^{n}$, the Whaldhausen suspension of $\D_{\bullet}^{n}$. We express the right formulation in the following corollary:\\

\textbf{Corollary F: }(relation to algebraic $\K$-theory \ref{corwaldhausen} )  
\textit{For any fibrant object $\D^{\bullet}_{\bullet}$ in $\Spn(\sCat_{\ast},\Sigma)$, we have the following  isomorphisms in $\Ho(\sSet)$: }
 \begin{enumerate}
 \item $\map_{\sCat_{\ast}}(\Sigma S^{0},\D^{n+1}_{\bullet})\simeq\map_{\sCat_{\ast}}(S^{0},\Omega\D^{n+1}_{\bullet})\simeq\diag\N\iso\Omega\D^{n+1}_{\bullet}.$

\item  $\map_{\sCat_{\ast}}(\Sigma S^{0},\D^{n+1}_{\bullet})\simeq \Omega \diag\N\iso\D^{n+1}_{\bullet}.$

\item  $\map_{\sCat_{\ast}}(\Sigma S^{0},\D^{n+1}_{\bullet})\simeq\map_{\sCat_{\ast}}(S^0,\D^{n}_{\bullet})\simeq\diag\N\iso\D^{n}_{\bullet}.$\\\\
 \end{enumerate}

The appendix is about some easy facts about small objects in different categories. We made an effort to treat all technical details in order to establish our results without ambiguity.\\\\

\textbf{Acknowledgment: } I would like to express my gratitude to my former supervisor, Professor  Kathryn Hess for all here corrections and suggestions.\\

\section{Diagonal model structures on $\sCat$}\label{section1}

 \subsection{framework}
 \begin{notation}
 \begin{enumerate}
 \item Objects of the category $\Cat$ are denoted by $\A, \B, \C,\dots$
 \item  Objects of the category $\sSet$ are denoted by $X_{\bullet},Y_{\bullet},Z_{\bullet}\dots$ or simply by  $X,Y,Z\dots$ if there is no confusion. 
 \item Objects of $\sCat$ will be denoted by $\A_{\bullet},\B_{\bullet},\C_{\bullet}\dots$ 
 \item  Generic categories by $\mathcal{A},\mathcal{B},\mathcal{C}\dots$ 
 \item Finally objects of the category of (non symmetric) spectra $\Spn({\sCat_{\ast}})$ are denoted by $\A_{\bullet}^{\bullet},\B_{\bullet}^{\bullet},\C_{\bullet}^{\bullet}\dots$
 \end{enumerate}
 \end{notation}
 In all what will follow, we assume that the category of small $\Cat$ is equipped with Joyal-Therny model structure. Roughly speaking, the weak equivalences are equivalences of categories, the cofibrations are functors injective on the set of objects and the fibrations are Grothendieck iso-fibrations. This model structure is in fact simplicial monoidal close cofibrantly generated model stucture, where all object are fibrant and cofibrant and consequently proper but not cellular.\\
 We use the standard Quillen model structure on the category of simplicial sets $\sSet$. The weak equivalences are morphisms which induce isomorphisms on homotopy groups, cofibrantions are monomorphisms and fibrations are Kan fibrations. Equipped with this model structure, the category of simplicial sets is simplicial monoidal closed cofibrantly generated model category, where all objects are cofibrant and Kan complexes are the fibrant objects. This model structure is proper and cellular. \\
In order to construct a model structure on the category of simplicial categories $\sCat=[\Delta^{op}, \Cat]$ we use the fundamental lemma of transferring model structure via an adjunction. 
 
%lemme fondamental de transfère de structure modèle par adjonction
  \begin{lemma}\label{lem1}[\cite{worytkiewicz2007model}, proposition 3.4.1]
Let $\mathcal{M}$ be a cofibrantly generated model structure and let  
$$\xymatrix{
\mathcal{M} \ar@<1ex>[r]^-{G} & \mathcal{C}\ar@<1ex>[l]^-{F}
}$$ be an adjunction and define a class of weak equivalences and fibrations as follow:
\begin{enumerate}
\item $\mathrm{WE}$ the class of  morphisms in $\mathcal{C}$ such that there image under $F$ is a weak equivalence in $\mathcal{M}$.
\item $\mathrm{Fib}$ the class of  morphisms in $\mathcal{C}$ such that there image under $F$ is a fibration in $\mathcal{M}$.
\end{enumerate}
Suppose that the following condition are verified:
\begin{enumerate}
\item The domains of $G (i)$ are small relatively to $G (\mathrm{I} )$ for all  $i \in \mathrm{I} $ and the domains of  $G(j)$ are small relatively to 
$G(\mathrm{J} )$ for all $ j \in \mathrm{J}$.
\item The functor $F$ commutes with directed  colimits i.e.,  
$$F\colim (\lambda \rightarrow \mathcal{C})= \colim F( \lambda\rightarrow\mathcal{C}).$$
\item Any transfinite composition of weak equivalences in $\mathcal{M}$ is again a weak equivalence.
\item The pushout of $G(j)$ along any morphism $f$ in  $\mathcal{C}$ is an element of $\mathrm{WE}.$
\end{enumerate}
Then $\mathcal{C}$ form is model category, where weak equivalences  (resp. fibrations) are  $\mathrm{WE}$  $(resp. \mathrm{Fib})$, moreover $\mathcal{C}$ is cofibrantly generated, where the generating cofibration
are $G(\mathrm{I})$ and generating trivial cofibrations are $G(\mathrm{J})$.
 \end{lemma}
 
 %%%%%%%%

\subsection{Diagonal model structure}
\begin{definition}
The  $ \pi:\sSet\rightarrow \Cat$ which associate to a simplicial set $K_{\bullet}$ its fundamental groupoid $\pi(K_{\bullet})$, where the objects are 0-simplicies $K_{0}$, and the generating isomorphisms are  $t:d_{1}x\rightarrow d_{0}x$ for each 1-simplex $t$ in $K_{1}$. The generators are submitted to the relation  $d_{0}l\circ d_{2}l=d_{1}l$ for all 2-simplices  $l$ in $K_{2}$.
\end{definition}
The functor $\pi$ admits a right adjoint $\N\iso$ which associates to $\C$ the nerve of the underlying  groupoid $\iso\C$.

The previous adjunction 
$$\xymatrix{
\sSet \ar@<1ex>[r]^-{ \pi} & \Cat\ar@<1ex>[l]^-{\N\iso}
}$$

extends naturally to an adjunction of bisimplicial sets and the category of simplicial objects in $\Cat$. $$\xymatrix{
\sSet^{2} \ar@<1ex>[r]^-{ \pi_{\bullet}} & \sCat\ar@<1ex>[l]^-{\N\iso}
}$$
where  $\pi_{\bullet}(K_{\bullet ,\bullet})_{n}=\pi(X_{\bullet,n})$ and $\N\iso$ is applied level-wise. 

%l'adjonction fondamentale entre sSet^2 et sCat
Moerdijk proved that there is a model structure on the category of bisimplicial sets Quillen equivalent to the standard model structure on simplicial sets  (\cite{goerss1999}, chapter 4, section 3). This model structure is obtained by the following adjunction:
$$\xymatrix{
\sSet\ar@<1ex>[r]^-{d_{\ast}} & \sSet^{2}.\ar@<1ex>[l]^-{\diag}
}$$

\begin{theorem}\label{structure}
 The category of simplicial categories $\sCat$ is a cofibrantly generated model category where

\begin{enumerate}
\item a morphism $f:\C_{\bullet}\rightarrow\D_{\bullet}$ 
is a weak equivalence (resp. fibration) if and only if  $\diag\N\iso(f)$ is a weak equivalence (resp. fibration) in $\sSet$.
\item The generating (acyclic)  cofibrations in $\sCat$ are given by the image of generating (acyclic) cofibration in $\sSet$ via the fonctor $\pi_{\bullet} d_{\ast}$. 
\end{enumerate}
\end{theorem}
 
\begin{lemma}\label{cofgenscat}
If $j$ is an generating acyclic cofibration in $\sSet^{2}$, then $\pi_{\bullet}(j)$ is an equivalence in $\sCat$.
\end{lemma}
\begin{proof}
The generating acyclic cofibrations in $\sSet^{2}$ equipped with Moerdijk's model structure are given by 
$$d_{\ast}\Lambda^{n}_{i}\rightarrow d_{\ast}\Delta^{n}=\Delta^{n,n},~i\in\{0,1,\dots, n\}.$$
More precisely:
$$\bigsqcup_{\beta\in\Lambda^{n}_{i}} C_{\beta}\rightarrow \bigsqcup_{\beta\in\Delta^{n}} \Delta^{n},$$
where $C_{\beta}$ is a contractible subcomplex of $\Lambda^{n}_{i}$.
Consider the following commutative diagram:

 $$\xymatrix{
  \bigsqcup_{\beta\in\Lambda^{n}_{i}} C_{\beta}  \ar[r]^-{pr}\ar[d]_-{d_{\ast}j}  & \bigsqcup_{\beta\in\Lambda^{n}_{i}}\ast \ar[d]^-{j} _{\thicksim} \\
    \bigsqcup_{\beta\in\Delta^{n}} \Delta^{n}\ar[r]^-{pr} & \bigsqcup_{\beta\in\Delta^{n}} \ast
  }$$
The projections are weak equivalences of simplicial sets degree by degree, and so diagonal equivalences. Obviously, we have also that $j$ is a diagonal equivalence. We conclude that $d_{\ast}j$ is a diagonal equivalence.
We Apply to the previous diagram the  functor $\N\pi$:
$$\xymatrix{
  \bigsqcup_{\beta\in\Lambda^{n}_{i}} \N\pi C_{\beta}  \ar[r]^-{pr}\ar[d]_-{ \N\pi_{\bullet} d_{\ast}j}  & \bigsqcup_{\beta\in\Lambda^{n}_{i}}\ast \ar[d]^-{j} _{\thicksim} \\
    \bigsqcup_{\beta\in\Delta^{n}}  \N\pi\Delta^{n}\ar[r]^-{pr} & \bigsqcup_{\beta\in\Delta^{n}} \ast
  }$$
Since $C_{\beta}$ is a connected subcomplex of  $ \Lambda^{n}_{i} $, the canonical projection $\pi C_{\beta}\rightarrow \ast$ is an equivalence of categories and induces an equivalence of nerves. Consequently, the horizontal arrows are equivalence degree wise, and so diagonal equivalences. We conclude that  $\N\pi _{\bullet} d_{\ast}(j)$ is also a diagonal equivalence. Finally, $\pi_{\bullet}(j)$ is a weak equivalence in $\sCat.$
\end{proof}
%lemme foncamental scat.
\begin{definition}
Let $\mathcal{M}$ be a category with a class of weak equivalences. A commutative square
$$\xymatrix{
    A \ar[r] \ar[d]  & C \ar[d] \\
    B \ar[r] & D
   }
  $$
in $\mathcal{M}$ is homotopically cocartesian if  the universal morphism $B\bigsqcup_{A}C\rightarrow D$  is a weak equivalence in $\mathcal{M}$ .
 
\end{definition}

\begin{lemma}\label{Niso}
Let $f:\A\rightarrow \B$ a fully faithful inclusion of groupoids. Consider the following pushout diagram in $\Cat:$
$$\xymatrix{
    \A \ar[r] \ar[d] ^{f} & \C \ar[d] \\
    \B \ar[r] & \D=\B\sqcup_{\A}\C.
   }
  $$
  Then the two diagrams 
  $$\xymatrix{
    \A=\iso\A \ar[r] \ar[d]^-{\iso f} &\iso \C \ar[d]  & & \N\iso \A=\N \iso\A \ar[r] \ar[d]^-{\N\iso f} &\N\iso \C \ar[d]\\
    \B=\iso\B \ar[r] & \iso\D && \N\iso\B=\N\iso\B \ar[r] & \N\iso\D
  }
  $$
 % and 
 %  $$\xymatrix{
   %\N\iso \A=\N \iso\A \ar[r] \ar[d]^-{\N\iso f} &\N\iso \C \ar[d] \\
    %\N\iso\B=\N\iso\B \ar[r] & \N\iso\D
  %}
  %$$ 
are homotopically cocartesian in $\Cat$ and respectively in $\sSet$. 
\end{lemma}
\begin{proof}
The hypothesis on the groupoids $\A$ and $\B$ imply that we can decompose in $\B^{1}\sqcup \B^{2}$  such  $f:\A\rightarrow\B_{1}$ is a trivial cofibration in $\Cat$. So, $\D=( \C\sqcup_{\A}\B_{1})\sqcup \B^{2}$. Lets define $\D^{'}=\B^{1}\sqcup_{\A}\C$, then $\D=\D^{'}\sqcup \B^{2}$. The functor $\C\rightarrow \C\sqcup_{\A}\B_{1}$ is an equivalence of categories, injective on objects. It follows that $\iso \C\rightarrow \iso(\C\sqcup_{\A}\B^{1})$ is a weak equivalence, consequently, the induced functor  $\iso \C\sqcup_{\A}\B_{1}\rightarrow \iso(\C\sqcup_{\A}\B^{1})$ is a weak equivalence in $\Cat$. Finally the second induced functor  
$$\iso\C\sqcup_{\A}\B=\iso \C\sqcup_{\A}\B^{1}\sqcup\B^{2}\rightarrow \iso(\C\sqcup_{\A}\B^{1}\sqcup\B^{2})=\iso \D$$
is an equivalence of categories.

Now, we apply the functor  $\N\iso$ to the initial diagram:
  
  $$\xymatrix{
  \N\A \ar[r] \ar[d] & \N\iso\C \ar[d]^{s}\ar@/^/[rdd] ^{l} \\
    \N\B\ar[r]  \ar@/_/[rrd]  & \N\B\sqcup_{\N\A}\N\iso\C  \ar@{.>}[rd]^{t} \\
     & &  \N\iso\D
  }$$
 We observe that  $\N\B\sqcup_{\N\A}\N\iso\C=(\N\B^{1}\sqcup_{\N\A}\N\iso\C)\sqcup \N\B^{2}$, and $\N\A\rightarrow \N\B^{1}$ is a trivial cofibration by definition of $\B^{1}$. So, the morphism $ \N\iso\C\rightarrow \N\B^{1}\sqcup_{\N\A}\N\iso\C$ is a weak equivalence, since $\sSet$ is a model category.
 On the other hand, $\N\iso\C\rightarrow \N\iso(\B^{1}\sqcup_{\A}\C)$ is a weak equivalence in $\sSet$, because it is the nerve of an equivalence of categories. Consequently, the induced map $\N\B^{1}\sqcup_{\N\A}\N\iso\C\rightarrow \N\iso(\B^{1}\sqcup_{\A}\C)$ is a weak equivalence of simplicial sets. We conclude that morphism of simplicial sets
 $$t:\N\B\sqcup_{\N\A}\N\iso\C=\N\B^{2}\sqcup\N\B^{1}\sqcup_{\N\A}\N\iso\C\rightarrow \N\B^{2}\sqcup\N\D^{'}= \N\iso\D$$ 
 is a weak equivalence of simplicial sets. 
\end{proof}

 \begin{lemma}\label{lemfond}
Let $j:\A\rightarrow \B$ be a generating acyclic cofibration in $\sCat$. Then the pushout of  $j$ a long any morphism $\A\rightarrow \C$  is a weak equivalence.
\end{lemma}

\begin{proof}
First of all, We remark that any acyclic generating cofibration in $\sCat$ verify degree by degree the hypothesis of lemma \ref{Niso}.
Consider the following pushout $\sCat:$ 
$$\xymatrix{
  \A \ar[r] \ar[d] ^{j} & \C \ar[d]^{j^{'}}\\
   \B\ar[r] & \D 
}$$
 
 Applying the functor $\diag\N\iso$ to the precedent pushout, we obtain a commutative diagram in  $\sSet$:
 
$$\xymatrix{
  \diag\N\A \ar[r] \ar[d] ^-{\diag\N\iso(j)}& \diag\N\iso\C \ar[d]^{s}\ar@/^/[rdd] ^{l} \\
   \diag \N\B\ar[r]  \ar@/_/[rrd]  & \diag(\N\B\sqcup_{\N\A}\N\iso\C)  \ar@{.>}[rd]^{\diag (t)} \\
     & &  \diag\N\iso\D,
  }$$
  
since the colimits in $\sCat$ are computed degree wise. Applying lemma \ref{Niso} degree by degree, we have that $t$ is a degree wise equivalence in $\sSet^{2}$, and so $\diag(t)$ is a weak equivalence in $\sSet$. On an other hand, $\diag\N\iso(j)$ is a cofibration of simplicial sets, since $\N\iso(j)$ is a degree wise monomorphism and it is a diagonal weak equivalence by \ref{cofgenscat}. Consequently, $s$ is a weak equivalence. By the property 2 out of 3, we conclude that $l$ is a weak equivalence.
  
 \end{proof}

Finally, we can prove that $\sCat$ is a cofibrantly generated model category

\begin{proof}  [\textbf{Proof of the Theorem} \ref{structure}] 
 The lemma \ref{lem1} permits to conclude that $\sCat$ is a cofibrantly generated model category since,
\begin{enumerate}
\item The hypothesis (1) is a consequence of  \ref{petitscat}.
\item The hypothesis (2) is a consequence of   \ref{smallscat}.
\item The hypothesis (3) is a consequence of the fact that in $\sSet$ a transfinite composition of weak equivalences is again a weak equivalence.
\item The hypothesis (4) is a consequence of \ref{lemfond}.
\end{enumerate}
 
\end{proof}

% Propriété de la structure modèle de $\sCat$
\section{Additional properties of the diagonal model structure on $\sCat$}\label{section2}
In this section, we will establish some properties of the model structure in $\sCat$, such that left and right properness and cellularity. 

\subsection{Cofibrations in $\sCat$}
In this paragraph, we describe some properties of cofibrations in $\sCat$ in order to prove that new model category is left proper. 
%We keep the same notation for the generating cofibration ($i: \alpha\rightarrow \beta$).
The simplicial set $\partial\Delta^{n}$  is generated by  $(d_{i}e_{n}$ avec $ 0 \leq i\leq n$ where $e_{n}$ is the only non degenerated  $n-$simplex of $\Delta^{n}$. The bisimplicial set $d_{\ast}\partial\Delta^{n}$ is generated by 
$(d_{i}e_{n},d_{i}e_{n})$. Following the same strategy as in (\cite{goerss1999}, chapter 4, 3.3), $d_{\ast}\partial\Delta^{n}$ can be described as 
$$ \bigsqcup _{\sigma\in \partial\Delta^{n} }C^{\sigma},$$
where the simplicial set $C^{\sigma}$ is generated by the faces $d_{i}e_{n}$ which contain $\sigma.$ Notice that the number of faces with this property is strictly less than $n+1$, this implies that $C^{\sigma}\rightarrow \ast$ is a weak equivalence. Moreover, $\pi(C^{\sigma})$ is a groupoid equivalent to the trivial groupoid $\ast$, more precisely, between to objects of the category $\pi(C^{\sigma})$ there is exactly one isomorphism.

%propriété fondamentale des cofibtaions dans sCat
Now, we give the fundamental property of cofibrations in $\sCat$ 
\begin{lemma}\label{cofinj}
The cofibrations in $\sCat$ are inclusions of categories i.e.,  the cofibrations are injective on objects and morphisms.
\end{lemma}
\begin{proof}
Let $i:\A_{\bullet}\rightarrow \B_{\bullet}$ be a generating cofibration in $\sCat$. Degreewise, the generating cofibrations in $\sCat$ have the property that they are inclusions of categories. Moreover, we have a decomposition of $\B_{n}$ in $\B_{n}^{1}\sqcup \B_{n}^{2}$ such that $i_{n}: \A_{n}\rightarrow \B_{n}^{1}$ is a trivial cofibration in $\Cat$. But all objects in $\Cat$ are fibrant, this implies that we have a retraction. Consequently the pushout of $i:\A_{\bullet}\rightarrow \B_{\bullet}$ along any functor $f:\A_{\bullet}\rightarrow\C_{\bullet}$ is an inclusion degree by degree, since the colimits in $\sCat$ are computed degreewise.
The transfinite composition of inclusions in $\sCat$ is again an inclusion and the $I-Cell$ are inclusion of categories.
By the same way the retracts of  $I-Cell$ are also inclusion of categories. We conclude that cofibrations in $\sCat$ are inclusions of categories. 
\end{proof}
   
  \begin{lemma}
Let  $i: \A_{\bullet}\rightarrow \B_{\bullet}$ be a cofibration in $\sCat$,
and consider the following pushout diagram in $\sCat$:
$$\xymatrix{
    \A_{\bullet}\ar[r] \ar[d] ^{i} & \C_{\bullet} \ar[d] \\
   \B_{\bullet} \ar[r] & \D_{\bullet}.
  }$$
  Then the functor  $\N\iso$ sends this pushout to a homotopically cocartesian square in  $\sSet^{2}$ equipped with projective model structure.  
  \end{lemma}
  \begin{proof}
  The pushouts in $\sCat$ are computed degreewise. The cofibrations in $\sCat$ verify the same hypothesis of  \ref{Niso} degreewise.
  \end{proof}
  
 \begin{remark}\label{remark1}
Consider a commutative diagram in a category $\mathcal{M}$:
  $$\xymatrix{
    A \ar[r] \ar[d] ^{f} & C \ar[d]\ar[r] & C^{'}\ar[d]\\
    B \ar[r] & D\ar[r] & D^{'}
  }$$
  where both squares are pushouts. Then the following square 
 $$\xymatrix{
    A \ar[r] \ar[d] ^{f}  & C^{'}\ar[d]\\
    B \ar[r] & D^{'}
  }$$
  is also a pushout diagram in $\mathcal{M}$.\\
  \end{remark}

  \begin{lemma}\label{Niso2}
  With the same notations as in \ref{remark1}, and $\A\rightarrow \B$ verifying the same hypothesis as in \ref{Niso}, then the natural morphisms
  $$\N\B\sqcup_{\N\iso\A}\N\iso\C^{'}\rightarrow\N\D\sqcup_{\N\iso\C}\N\iso\C^{'}$$
    %$$  (\N\B\sqcup_{\N\A}\N\iso\C)\sqcup_{\N\iso\C}\N\iso\C^{'}\rightarrow \N\iso\D\sqcup_{\N\iso\C}\N\iso\C^{'}}$$
  and 
  $$\N\iso\D\sqcup_{\N\iso\C}\N\iso\C^{'}\rightarrow \N\iso\D^{'}$$
are weak equivalences in $\sSet$.
  \end{lemma}
  \begin{proof}
  By \ref{Niso}, the morphism $\N\B\sqcup_{\N\A}\N\iso\C\rightarrow \N\iso \D$ is an equivalence. Moreover
  $\N\iso \C\rightarrow \N\B\sqcup_{\N\A}\N\iso\C$ and the morphism $\N\iso\C\rightarrow \N\iso\D$ are cofibrations in $\sSet$.
Since $\sSet$ is left proper, we conclude that  
  $$(\N\B\sqcup_{\N\A}\N\iso\C)\sqcup_{\N\iso\C}\N\iso\C^{'}\rightarrow \N\iso\D\sqcup_{\N\iso\C}\N\iso\C^{'}$$
is a weak equivalence.
  to show the other equivalence, it is sufficient to remark that 
  $$(\N\B\sqcup_{\N\A}\N\iso\C)\sqcup_{\N\iso\C}\N\iso\C^{'}=\N\B\sqcup_{\N\A}\N\iso\C^{'}\rightarrow  \N\iso\D^{'}$$
 is an equivalence by \ref{Niso} and by the property  "2 out of 3"
  $$\N\iso\D\sqcup_{\N\iso\C}\N\iso\C^{'}\rightarrow\N\iso\D^{'}$$ is a weak equivalence.
  \end{proof}

  \begin{corollary}\label{Niso3}
  The simplicial version of lemma \ref{Niso2} if we replace $f:\A\rightarrow\B$ by 
  $i:\A_{\bullet}\rightarrow\B_{\bullet}$ in $\sCat$ verifying the same hypothesis degrewise, then 
  $$  \N\B_{\bullet}\sqcup_{\N\A_{\bullet}}\N\iso\C^{'}_{\bullet}\rightarrow \N\iso\D_{\bullet}\sqcup_{\N\iso\C_{\bullet}}\N\iso\C^{'}_{\bullet}$$
  and 
  $$\N\iso\D_{\bullet}\sqcup_{\N\iso\C_{\bullet}}\N\iso\C^{'}_{\bullet}\rightarrow \N\iso\D^{'}_{\bullet}$$
  are degreewise weak equivalences.
  \end{corollary}
  \begin{proof}
  Apply  \ref{Niso2} degree by degree.
  \end{proof}

 \subsection{Properness of $\sCat$ }
  We prove that  $\sCat$ equipped with the model structure of \ref{structure} is left proper. 
  \begin{lemma}\label{Niso4}
  Let $i:\A_{\bullet}\rightarrow\B_{\bullet}$ an element of $I-Cell$ in $\sCat$. The functor $\diag\N\iso$ sends the following pushout in $\sCat$:
  $$\xymatrix{
    \A_{\bullet} \ar[r] \ar[d] ^{i} & \C_{\bullet} \ar[d] \\
    \B_{\bullet} \ar[r] & \D_{\bullet}
  }$$
 to a homotopically cocartesian square in $\sSet$.
  \end{lemma}
  \begin{proof}
  First of all, $\A_{\bullet}\rightarrow \B_{\bullet}$ is a transfinite composition of cofibrations of the form 
  $$ \A_{\bullet}=\C^{0}\rightarrow\dots \A^{s}_{\bullet}\rightarrow\A_{\bullet}^{s+1}\rightarrow\A^{s+2}\rightarrow\dots $$
  We denote $\A^{s}_{\bullet}\sqcup_{\A_{\bullet}}\C_{\bullet}$ by $\C_{\bullet}^{s}$. 
  By the corollary \ref{Niso3}
  $$\N\iso\A_{\bullet}^{s}\sqcup_{\N\iso \A_{\bullet}}\N\iso\C_{\bullet}\rightarrow\N\iso\C_{\bullet}^{s}$$ 
  is a weak equivalence, moreover,  $\N\iso\C_{\bullet}^{s}\rightarrow\N\iso\C^{s+1}_{\bullet}$ is an inclusion of bisimplicial stes. Knowing that $\N\iso$ commutes with directed colimits, and that \\
  $\diag\N\iso\C_{\bullet}^{s}\rightarrow\diag\N\iso\C_{\bullet}^{s+1}$
  is a cofibration in $\sSet$, we conclude that :
  $$ \diag(\N\iso\B_{\bullet}\sqcup_{\N\iso \A_{\bullet}}\N\iso\C_{\bullet})\rightarrow \diag\N\iso \D_{\bullet}$$
is an equivalence in $\sSet$
   \end{proof}
   \begin{corollary}\label{Niso5}
  Let $i:\A^{'}_{\bullet}\rightarrow\B^{'}_{\bullet}$ be any cofibration in $\sCat$. 
  The functor  $\diag\N\iso$ sends the following pushout in $\sCat$: 
  $$\xymatrix{
    \A^{'}_{\bullet}\ar[r] \ar[d] ^{i} & \C_{\bullet} \ar[d] \\
    \B^{'}_{\bullet} \ar[r] & \D_{\bullet}
  }$$
  to a homotopy cocartesian square.
   \end{corollary}
   \begin{proof}
 The cofibration $i^{'}:\A^{'}_{\bullet}\rightarrow\B^{'}_{\bullet}$  is a retract of a $I-cell$ cofibration  $i:\A_{\bullet}\rightarrow\B_{\bullet}$. We denote  $\B_{\bullet}\sqcup_{\A_{\bullet}}\C_{\bullet}=\M_{\bullet}$, and $\B^{'}_{\bullet}\sqcup_{\A^{'}}\C_{\bullet}=\D_{\bullet}$. There is an induced retract:
  $$\xymatrix{
    \N\iso\B^{'}_{\bullet}\sqcup_{\N\iso\A^{'}_{\bullet}} \N\iso\C_{\bullet} \ar[r] \ar[d] & \N\iso\B_{\bullet}\sqcup_{\N\iso\A_{\bullet}} \N\iso\C_{\bullet}\ar[d]^{t}\ar[r]  & \N\iso\B^{'}_{\bullet}\sqcup_{\N\iso\A^{'}_{\bullet}} \N\iso\C_{\bullet} \ar[d]^{g} \\
  \N\iso\D_{\bullet} \ar[r] & \N\iso\M_{\bullet} \ar[r] & \N\iso\D_{\bullet}
  }$$
 By lemma \ref{Niso4}, we have that  $\diag(t)$ is an equivalence, so  $\diag(g)$ is also an equivalence in $\sSet$.
   
   \end{proof}
  
\begin{corollary}\label{leftproper}
The model category $\sCat$ is left proper. 
\end{corollary}
\begin{proof}
Let $\A_{\bullet}\rightarrow\B_{\bullet}$ be a cofibration and let  $f:\A_{\bullet}\rightarrow \C_{\bullet}$ be an equivalence in $\sCat$.
It is sufficient to consider the following pushout :
  $$\xymatrix{
  \diag\N\iso\A_{\bullet} \ar[r] ^{f}_{\sim}\ar[d]^{i} & \diag\N\iso\C_{\bullet} \ar[d]^{s}\ar@/^/[rdd] ^{l} \\
    \diag\N\iso\B_{\bullet}\ar[r]^-{g}  \ar@/_/[rrd]^{h}  & \diag(\N\iso\B_{\bullet}\sqcup_{\N\iso\A_{\bullet}}\N\iso\C_{\bullet} ) \ar@{.>}[rd]^{t} \\
     & &  \diag\N\iso\D_{\bullet}
  }$$
  We have that $i$is a cofibration in $\sSet$ since $\N\iso\A\rightarrow\N\iso\B$ is injective in $\sSet^{2}$. But $f$ is an equivalence and $\sSet$ is proper, this implies that $g$ is an equivalence.
By the corollary \ref{Niso5}, $t$ is a weak equivalence, this implies that $h$ is a weak equivalence and finally that $\sCat$ is left proper. 
\end{proof}

    \begin{lemma}
    The model category $\sCat$ is right proper.
    \end{lemma}
    
    \begin{proof}
    This fact is much more easier that the left properness.  Consider the pullback diagram in $\sCat$:
    $$\xymatrix{
    \C_{\bullet}\times_{\mathbf{A}}\D_{\bullet}\ar[r]^-{i}\ar[d]  & \C_{\bullet}\ar@{>>}[d]_-{f}     \\
    \D_{\bullet} \ar[r]_-{j}^{\thicksim} & \mathbf{A}_{\bullet}
  }$$
  Our goal is to show that  $i$ is a weak equivalence. Applying the functor $\diag~\N \iso$ which commutes with limits, we  obtain a pullback diagram in $\sSet$, such that
    $\diag\N\iso (f)$ is a fibration by definition of model structure on $\sCat$ . Since $\sSet$ is right proper, we conclude that $\diag~\N\iso (i)$ is also a weak equivalence, and so $\sCat$ is right proper.
        \end{proof}

     %cellularité

    \subsection{Cellularity of $\sCat$}
    
    This paragraph is an other step of our comprehension of cofibrations in the model category $\sCat$. We will prove the \textit{cellularity} property. This step is crucial in order to consider the left Bousfield localization and stabilization of  the model category $\sCat.$ 
    \begin{definition}\label{cell}
    A cofibrantly generated model category is cellular if it verifies the following conditions:
    
    \begin{enumerate}
    \item the domains and codomains of $\mathrm{I}$ are compact (cf \cite{Hirs} 11.4.1);
    \item the domain of the generating acyclic cofibrations $\mathrm{J}$ are small with respect to $ \mathrm{I}$  (cf \cite{Hirs} 10.5.12); and
    \item the cofibrations are effective monomorphisms (cf \cite{Hirs} 10.9.1).
    \end{enumerate}
    \end{definition}

     %condition 1 de cellularité
     
      \begin{lemma}
      The cofibrations and acyclic cofibrations of $\sCat$ verify the first hypothesis of the definition \ref{cell}.
      \end{lemma}
      \begin{proof}
      Let $\C_{\bullet}$ a domain (codomain) of an element in $\mathrm{I}$. By definition $\C_{\bullet}$ has the form $\pi~ d_{\ast}X_{\bullet}$, where $d_{\ast}$ is the left adjoint to $\diag.$
      Let $f~: \D_{\bullet}\rightarrow \D^{'}_{\bullet}$ be a $\mathrm{I}$-cell complex.We have to show that any morphism $g:~\C_{\bullet}\rightarrow \D^{'}_{\bullet}$ is factorized through a sub complex of $f$ for a certain cardinal  $\gamma$. The morphism $f$ is a transfinite composition of elements of $\mathrm{I}-cell$ which are inclusions categories degree by degree
      $$ \D_{\bullet}\rightarrow \D_{\bullet}^{1}\dots \D_{\bullet}^{\beta}\rightarrow\D_{\bullet}^{\beta+1}\dots\rightarrow \D_{\bullet}^{'}.$$
      The factorization  $g$ by a sub complex of $f$ is equivalent to factorization of the adjoint of  $g$, denoted by,  $g^{'}~:d_{\ast}X_{\bullet}\rightarrow \N\iso \D_{\bullet}^{'}$ by a sub complex of 
      $$ \N\iso\D_{\bullet}\rightarrow\N\iso\D_{\bullet}^{1}\dots \N\iso\D_{\bullet}^{\beta}\rightarrow \N\iso \D_{\bullet}^{\beta+1}\dots\rightarrow \N\iso\D_{\bullet}^{'}$$
     By the same argument, a factorization of $g^{'}$ is equivalent to factorization of the adjoint map $g^{''}: X_{\bullet}\rightarrow
\diag \N\iso \D_{\bullet}^{'}$ by a sub complexe 
       $$ \diag \N\iso\D_{\bullet}\rightarrow \diag \N\iso\D_{\bullet}^{1}\dots \diag \N\iso\D_{\bullet}^{\beta}\rightarrow \diag \N\iso\D_{\bullet}^{\beta+1}\dots\rightarrow \diag \N\iso\D_{\bullet}^{'}$$
      which is a transfinite composition of monomorphisms in $\sSet$, since by \ref{cofinj} the cofibrations in $\sCat$ are inclusions of categories degree by degree.
      But the objects $\Delta^{n}$ and $\partial\Delta^{n}$ are compact in $\sSet$. We conclude that  $g$ has a factorization through a sub complex of $f$.
      
      \end{proof}

      %condition 2 cellularité
      
      \begin{lemma}
   The cofibrations and acyclic cofibrations in $\sCat$ the second hypothesis of the definition \ref{cell}.
      \end{lemma}
      \begin{proof}
      Let $\pi~d_{\ast}X_{\bullet}$ be a domain of an element of $\mathrm{J}$.
      We have to show that this domain is small relatively to $\mathrm{I}$-cell for a certain cardinal $\lambda$.
      We have the following isomorphisms:
      $$ \colim_{\beta<\lambda}\homs_{\sCat}(\pi d_{\ast}X_{\bullet}, \D_{\bullet}^{\beta})\rightarrow \colim_{\beta<\lambda}\homs_{\sSet}(X_{\bullet}, \diag \N\iso\D_{\bullet}^{\beta})$$
      $$\colim_{\beta<\lambda}\homs_{\sSet}(X_{\bullet}, \diag \N \iso\D_{\bullet}^{\beta})\rightarrow\homs_{\sSet}(X_{\bullet}, \diag \N \iso ~ \colim_{\beta<\lambda} \D_{\bullet}^{\beta})$$
      The first isomorphism is by adjunction, the second isomorphism is a consequence of the fact that all object in $\sSet$ are small for a  $\lambda$ (cf \cite{Hovey} lemme 3.1.1), and the functor $\diag$ 
      commutes with colimits  and that the functor $\N\iso$ commutes with directed colimits.
      \end{proof}

      %condition 3 cellularité

      \begin{lemma}
      The cofibrations in $\sCat$ are effective monomorphisms.
      \end{lemma}
      \begin{proof}
    Let $\xymatrix{ \C_{\bullet}  \ar@{^{(}->}[r]^-{i}  & \D_{\bullet}}$ be a  cofibration in $\sCat$ (in particular an inclusion of categories). Now, we compute the equalizer of the following diagram:
      $$\xymatrix{
\D_{\bullet}\ar@<1ex>[r] \ar[r] & \D_{\bullet}\bigsqcup_{\C_{\bullet}} \D_{\bullet}
}$$
where both morphisms are inclusion of categories coming from the pushout diagram:
   $$\xymatrix{
 \C_{\bullet} \ar@{^{(}->}[r]^-{i}\ar@{^{(}->}[d]^-{i} & \D_{\bullet}\ar@{^{(}->}[d]^{i_{1}}\\
  \D_{\bullet}\ar@{^{(}->}[r] _{i_{2}}& \D_{\bullet}\bigsqcup_{\C_{\bullet}} \D_{\bullet}.
     }$$
     We affirm that the equalizer is given exactly by  
     $$\xymatrix{
\C_{\bullet}\ar[r]^-{i} & \D_{\bullet}\ar@<1ex>[r]\ar[r] & \D_{\bullet}\bigsqcup_{\C_{\bullet}} \D_{\bullet}
}$$
First of all, it is a commutative diagram. Suppose that $\C^{'}_{\bullet}$ is an other condidat for the equalizer. Since the functor $\Ob:\sCat\rightarrow\sSet$ commutes with limits and colimits, there exists a unique morphism $t$ making the following diagram commuting:
$$\xymatrix{
\Ob\C_{\bullet}^{'}\ar@{.>}[d]^{t} \ar[rd]^{\Ob(F)} & \\
\Ob\C_{\bullet}\ar[r]^-{\Ob(i)}& \Ob\D_{\bullet}\ar@<1ex>[r] \ar[r] & \Ob\D_{\bullet}\bigsqcup_{\Ob\C_{\bullet}} \Ob\D_{\bullet}
}$$
In fact, the cofibrations in $\sCat$ are injective on objects \ref{cofinj}, and $\sSet$ is cellular \cite{Hirs}.
Suppose now $\gamma$ a morphism in teh category $\C^{'}_{\bullet}$ such that$i_{1}F(\gamma)=i_{2}F(\gamma)$. Since $i_{1}:\C_{\bullet}\rightarrow\D_{\bullet}\bigsqcup_{\C_{\bullet}}\D_{\bullet}$ and
$i_{2}:\C_{\bullet}\rightarrow\D_{\bullet}\bigsqcup_{\C_{\bullet}}\D_{\bullet}$ are injections of categories, it implies that  $F(\gamma)$ is, in fact, a morphism in $\C_{\bullet}$.
We conclude that any morphism $F:\C^{'}_{\bullet}\rightarrow \D_{\bullet}$ in $\sCat$ such that$i_{1}F=i_{2}F$ has a unique factorization as a composition
$$ \C^{'}_{\bullet}\rightarrow \C_{\bullet}\rightarrow \D_{\bullet}.$$
\end{proof}
      
 \begin{corollary}\label{cellsacat}
 Equipped with the model structure \ref{structure}, The category $\sCat$ is cellular.
 \end{corollary}  
  \begin{remark}
 At this stage, we should remark that the model category on $\Cat$  constructed by A. Joyal and refined by C. Rezk is not a cellular. In order to show the non cellularity of  $\Cat$, we consider the following example where $\C\rightarrow \D$ is a cofibration in $\Cat$ i.e.,    $$  \def\dar[#1]{\ar@<2pt>[#1]\ar@<-2pt>[#1]}
  \xymatrix{ a \dar[r] & b }$$
  and $\D$ the category with two objects $$a\rightarrow b$$
  In this case, the equalizer 
  $$  \def\dar[#1]{\ar@<2pt>[#1]\ar@<-2pt>[#1]}
  \xymatrix{ \D \dar[r] & \D\cup_{\C}\D }$$
  is $\D$ and  not $\C$, since $\D\cup_{\C}\D = \D.$\\
  
More over the suspension functor in $\Cat_{\ast}$ is trivial (equivalent to the identity functor). For all this raisons, it is not interesting to consider the category of spectra $\Spn(\Cat_{\ast})$.\\
\end{remark}
%La conclusion finale est que la catégorie $\sCat$ est une catégorie modèle propre de génération cofibrante et cellulaire. On peut donc lancer le processus de "stabilisation" décrit dans \cite{hovey2001}.

 \section{The $\overline{W}$-Model structure on $\sCat$ } \label{section3}
The goal of this section is to introduce a second new model structure on the category $\sCat$ Quillen equivalent to the previous. The second new model structure has all the good properties (proper, cellular) and more over it is (co)tensored over the model category of simplicial sets in a compatible way with the model structure. Our main inspiration come from the technical artical \cite{cegarra2007}. Roughly speaking, the authors use a new adjunction between $\sSet$ and $\sSet^{2}$ in order to transfer the model structre to the category of bisimplicial sets, it is denoted by $\overline{W}$-Model structure. The class of weak equivalences are the same as in the Moerdijk model structure on $\sSet^{2}$ but there is less cofibration and more fibrations. The left adjoint functor $Dec:\sSet\rightarrow\sSet^{2}$ used for defining 
 $\overline{W}$-Model structure is cartesian i.e., $Dec(X_{\bullet}\times Y_{\bullet})=Dec (X_{\bullet})\times Dec (Y_{\bullet})$. This observation is crucial for our propose. We will explain the consequence of such observation for the $\overline{W}$-Model structure on $\sCat$. \\

\begin{definition}
The Illusie functor $Dec: \sSet\rightarrow \sSet^{2}$ is defined for all simplicial sets $Y_{\bullet}$  by$Dec(Y_{\bullet})_{p,q}=Y_{p+q+1}~\forall p,q$ . The horizontal faces are given by $d^{h}_{i}=d_{i}:~Y_{p+q+1}\rightarrow Y_{p+q}$, in the same way the degeneracies are $s^{h}_{i}=s_{i}$. The vertical faces are given by  $d_{j}^{v}=d_{p+1+j}:~Y_{p+1+q}\rightarrow Y_{p+q}$ and the vertical degeneracies are$s_{j}^{v}=s_{p+1+j}$.
\end{definition}
\begin{lemma}
The functor $Dec$ has a right adjoint $\overline{W}:~\sSet^{2}\rightarrow \sSet$ defined as follow:
$$ \overline{W}(X)_{n}=\{(x_{0,n},\dots,x_{n,0})\in \prod_{p=0}^{n}X_{p,n-p}|~ d_{0}^{v}x_{p,n-p}=d^{h}_{p+1}x_{p+1,n-p-1},0\leq p< n \}$$
\end{lemma}
For the definition of degeneracies and faces of the simplicial set $\overline{W}(X)$ we refer to \cite{cegarra2007}.
\begin{corollary}\label{W-modele}
The functor  $\overline{W}$ commutes with directed colimites.
\end{corollary}
\begin{proof}
Let  $\colim_{\lambda}X$ be a directed colimit of objets in $\sSet^{2}$, the equality
$$\overline{W}(\colim_{\lambda}X)_{n}=\colim_{\lambda}\overline{W}(X)_{n}$$
is a consequence of the fact that a finite products commutes with directed colimits in $\sSet^{2}$.
\end{proof}

\begin{theorem} \label{W-structure}\cite{cegarra2007}
The category of bisimplicial sets $\sSet^{2}$ admits a structure of cofibrantly generated model category, denoted by $\overline{W}$-structure, where a morphism $f$ is a fibration (weak equivalence) if $\overline{W}(f)$ is a fibration (weak equivalence) of simplicial sets. Moreover:
\begin{enumerate}
\item any Moerdijk  fibration is a $\overline{W}$-fibration;
\item any $\overline{W}$-cofibration is a Moerdijk-cofibration;
\item a morphism of bisimplicial sets is a Moerdijk weak equivalence (i.e., diagonal equivalence) if and only if it is a $\overline{W}$-equivalence.  
\end{enumerate}
\end{theorem}

Moreover, the new $\overline{W}$-model structure on bisimplicial sets is cofibrantly generated, where 
\begin{enumerate}
 \item the generating cofibrations are given by  $Dec~ \partial\Delta^{n}\rightarrow Dec ~\Delta^{n}, ~n\in \mathbb{N}$.
 \item the generating acyclic cofibrations are given by  $Dec~ \Lambda^{n}_{i}\rightarrow Dec~ \Delta^{n}, ~n\in \mathbb{N}, ~ 0\leq i\leq n$.
\end{enumerate}

\begin{remark}
The $\overline{W}$-structure and Moerdijk model structure on bisimplicial sets are Quillen equivalente, the equivalence is given by the functor identity. Since the functor $Dec$ is cartesian, the $\overline{W}$-model structure on simplicial sets is (co)tensored (in a compatible way) over the model category à $\sSet$.
\end{remark}

\begin{theorem}\label{newstructure}
There exists a $\overline{W}$-model structure on $\sCat$ equivalent to the diagonal model structure \ref{structure}  induced by the adjunction 
$$\xymatrix{
\sSet \ar@<1ex>[r]^-{\pi_{\bullet}~Dec} & \sCat \ar@<1ex>[l]^-{\overline{W}\N\iso} 
}$$
\begin{enumerate}
\item A morphism $f:\C_{\bullet}\rightarrow \D_{\bullet}$ is a weak equivalence (resp. fibration)
if $\overline{W}\N\iso(f)$ is a weak equivalence (resp. fibration) of simplicial sets. 
\item The generating cofibrations are given by $\pi_{\bullet}Dec(\partial\Delta^{n})\rightarrow\pi_{\bullet}Dec(\Delta^{n})$ for all $n\in \mathbb{N}$.
\item The generating acyclic  cofibrations are given by $\pi_{\bullet}Dec(\Lambda^{n}_{i})\rightarrow\pi_{\bullet}Dec(\Delta^{n})$ for all $n\in \mathbb{N}$ and $0\leq i\leq n$. 
\end{enumerate}
 Moreover, the $\overline{W}$-model structure on $\sCat$ is left proper and cellular.
\end{theorem}

\begin{proof}

The generating cofibrations are given by  $\pi_{\bullet} Dec~ \partial\Delta^{n}\rightarrow \pi_{\bullet} Dec ~\Delta^{n}, ~n\in \mathbb{N}$.
The generating acyclic cofibrations are given by  $\pi_{\bullet} Dec~ \Lambda^{n}_{i}\rightarrow \pi_{\bullet} Dec~ \Delta^{n}, ~n\in \mathbb{N}, ~ 0\leq i\leq n$.
In order to show that this choice of (acyclic) cofibrations determines a model structure, it  is sufficient to show the hypothesis (2) and (4) from lemma \ref{lem1}.
The point (2) is a direct consequence of \ref{W-modele}.
Let  $j: Dec~ \Lambda^{n}_{i}\rightarrow Dec~ \Delta^{n}$ be generating acyclic cofibration in $\overline{W}$-model structure
$\sSet^{2}$. We know by \ref{W-structure} that $j$ is an acyclic cofibration in the Moerdijk diagonal model structure on $\sSet^{2}$. Consequently, $\pi_{\bullet}(j)$ is an acyclic cofibration in the diagonal model structure on  
$\sCat$ \ref{structure}. So, the pushout of $\pi_{\bullet}(j)$ along a morphism  $f:\pi_{\bullet} Dec~ \Lambda^{n}_{i}\rightarrow\C_{\bullet}$ in  $\sCat$:
 $$\xymatrix{
   \pi_{\bullet} Dec~ \Lambda^{n}_{i} \ar[r]^{f} \ar[d] ^{\pi_{\bullet}(j)} & \C_{\bullet} \ar[d] \\
    \pi_{\bullet} Dec~ \Delta^{n} \ar[r] & \D_{\bullet}
  }$$
is a weak diagonal equivalence i.e., $\diag\N\iso \C_{\bullet}\rightarrow\diag\N\iso \D_{\bullet}$ is a weak equivalence in $\sSet$. By \ref{W-structure}, we conclude that 
$$\overline{W}\N\iso \C_{\bullet}\rightarrow\overline{W}\N\iso \D_{\bullet}$$
is an equivalence in $\sSet$.\\
To show that $\overline{W}$-model structure on $\sCat$ is left proper and cellular, we remark that  cofibrations in $\overline{W}$-model structure on $\sCat$ are also cofibrations  in the diagonal model structre on $\sCat$. Consequently, we have less cofibrations in
 $\overline{W}$-model structure $\sCat$ than in the diagonal model structure on  $\sCat$, but in the same time the class of weak equivalences are the same in both model structures, it implies that  $\overline{W}$-model structure on $\sCat$ is left proper, cellular and Quillen equivalent to \ref{structure}.\\
 
\end{proof}
\begin{remark}
We should remark at this stage that  the $\overline{W}$-model structure on $\sCat$ is deduced form the diagonal model structure on $\sCat$. It seems that a direct proof of $\overline{W}$-model structure is quite hard. The following section we will see why it is better to consider $\overline{W}$-structure than the diagonal one.
\end{remark} 
\begin{remark}
The diagonal model structre (resp. the $\overline{W}$-model structure) on $\sCat$ can be restricted in a natural way to the diagonal model structure (resp. $\overline{W}$-model structure) on $\mathbf{sGrp}$, the category of simplicial groupoids.  
\end{remark}

%%%%%%%%%%%%%%%%%%%%%%%%%%%%%%%%%%%%%%

\section{Pointed model structure $\sCat_{\ast}$ }\label{section4}

The main goal of this section is to define the suspension and loop functor in the model category $\sCat$. In order to construct such functors we need a pointed model version of $\sCat$. We denote the pointed category by $\sCat_{\ast}$ or by $\ast\downarrow \sCat$   (\cite{Hovey}, chapter 6).

\begin{definition}
A pointed category $\C$ is equipped with a functor $\ast\rightarrow \C$  where $\ast$ is the terminal object in  $\sCat$.  
\end{definition}     
Recall the the tensorization of $\Cat$ by $\sSet$ is defined by
$$X_{\bullet}\otimes \C = \pi X_{\bullet}\times \C$$  
similarly, the cotensorization is given by 
$$ \C^{X_{\bullet}}= \HOM_{\Cat}(\pi X_{\bullet},\C)$$

We construct a (co)tensorization of $\sCat_{\ast}$ by $\sSet_{\ast}$ following the same procedures as before but in more general context.
Suppose that we have an adjunction between $\sSet$ and $\sCat$ such that the left adjoint is cartesian $\rho :~\sSet\rightarrow \sCat$ i.e., 
$\rho(X_{\bullet})\times \rho(Y_{\bullet})=\rho(X_{\bullet}\times Y_{\bullet}).$ 
The tonsororization is defined by  $\C_{\bullet}\otimes_{\rho} X_{\bullet}= \C\times \rho X_{\bullet}$ and the cotensorization by $\C^{X}= \HOM_{\sCat}(\rho X_{\bullet}, \C_{\bullet}).$

\begin{definition}
Let $\C_{\bullet}$ an object of $\sCat_{\ast}$. In order to construct the tensorization,  
$$ -\odot-:~\sCat_{\ast}\times\sSet\rightarrow \sCat_{\ast} $$ 
we start by defining the tensor product with $\Delta^{n}$ then for all simplicial sets  $\sSet$ by left Kan extension. In particular, $\C_{\bullet}\odot\Delta^{n}$ is given by the pushout: 

$$\xymatrix{ 
 \ast \otimes_{\rho}\Delta^{n}\ar[rr]\ar[d] & & \C_{\bullet}\otimes_{\rho}\Delta^{n} \ar[d]\\
 \ast \ar[rr] && \C_{\bullet}\odot_{\rho}\Delta^{n}.
 }$$

\end{definition}

\begin{definition}
The smash product  $-\land_{\rho}-:~\sSet_{\ast}\times\sCat_{\ast}\rightarrow \sCat_{\ast}$ is defined first for $\Delta_{+}^{n}$ by the formula
$$\C_{\bullet}\land_{\rho}\Delta_{+}^{n}=\C_{\bullet}\odot_{\rho}\Delta^{n}$$
and then extended to $\sSet_{\ast}$ by left Kan extension.

\end{definition}
 %adjonction 
\begin{lemma}\label{cotensrho}
The functor $-\odot_{\rho} X_{\bullet}:~ \sCat_{\ast}\rightarrow \sCat_{\ast}$ admits a right adjoint which is denoted by  $(-)^{X_{\bullet}}:~ \sCat_{\ast}\rightarrow \sCat_{\ast}$
\end{lemma}
\begin{proof}
First of all, we construct the adjoint for $-\odot_{\rho} \Delta^{n}$. The functor 
$$ \mathbf{HOM}_{\sCat}({\rho\Delta^{n}},-):\sCat_{\ast}\rightarrow\sCat_{\ast}$$
 which is the right adjoint to the cartesian product for $\sCat$, it sends a pointed category $\C_{\bullet}$ to a pointed category $\mathbf{HOM}_{\sCat}(\rho\Delta^{n}, \C_{\bullet})$   %$\HOM_{\sCat}(\rho \Delta^{n},\C_{\bullet})$
 , where the point is given by the constant functor $0:~\rho \Delta^{n}\rightarrow\C_{\bullet}.$
 We have to verify that it is an adjoint of $-\odot_{\rho}\Delta^{n}$ in $\sCat_{\ast}.$
 Giving a (simplicial) functor $f:~\C_{\bullet}\odot_{\rho} \Delta^{n}\rightarrow \D_{\bullet}$ is equivalent to give a functor  $\tilde{f}:\C_{\bullet}\times \rho \Delta^{n}\rightarrow \D_{\bullet}$ which sends the sub category  $\ast\otimes_{\rho}\Delta^{n}$ to the base point in
 $\D_{\bullet}$, by the same way, it is equivalent to give a pointed functor $g: \C_{\bullet}\rightarrow\HOM_{\sCat}(\rho \Delta^{n},\D_{\bullet}):=\D_{\bullet}^{\Delta^{n}}$.
So,
%$$\homs_{\sCat_{\ast}}(\C_{\bullet}\odot_{\rho} \Delta^{n},\D_{\bullet})= \homs_{\sCat_{\ast}}(\C_{\bullet},\D_{\bullet}^{\Delta^{n}}).$$
in order to prove the adjunction for any simplicial set $X_{\bullet}$, we remark that
$$ \C_{\bullet}\odot_{\rho} (\colim_{\Delta^{n}\rightarrow X_{\bullet}}\Delta^{n})=\colim_{\Delta^{n}\rightarrow X_{\bullet}}(\C_{\bullet}\odot_{\rho}\Delta^{n}).$$

 \end{proof}

Taking our inspiration from $\sSet_{\ast}$ we construct a new model structure on $\sCat_{\ast}$ using the adjunction induced by the forgetful functor and left adjoint which adds a base point. We will show that the new model structure on  $\sCat_{\ast}$ is (co)tensored over the model category of simplicial sets  $\sSet_{\ast}$. The adjunction 
$$\xymatrix{
\sCat \ar@<1ex>[r]^-{(-)_{+}} & \sCat_{\ast}\ar@<1ex>[l]^-{U}
}$$
defines a model structure on $\sCat_{\ast}$, the weak equivalences and fibrations are simply those in the underlying model category  $\sCat$. For more details see \cite{Hovey}.

In this paragraph, we show that for any pointed simplicial set $X_{\bullet}$, the functors
$-\land X_{\bullet} $ and  $ (-)^{X}$ form a Quillen pair.
First of all, that a simplicial pointed set $p:~\Delta^{0}\rightarrow X$ by the following pushout diagram:

$$\xymatrix{
\Delta^{0}_{+}\ar[r]^{p_{+}}\ar[d] & X_{+}\ar[d]\\
\ast\ar[r] & X
}$$

\begin{theorem}\label{S-Quillen}
If $X_{\bullet}$ is a pointed simplicial set, then the functor
$$-\land X_{\bullet} :~ \sCat_{\ast}\rightarrow \sCat_{\ast}$$
 is a left Quillen functor, where $\sCat$ is equipped with $\overline{W}$-model structure. 
 Moreover, the functor $-\land X_{\bullet} $ has a right Quillen adjoint. 

\end{theorem}

\begin{proof}
First, we simplify our notation, a simplicial set $X_{\bullet}$ will be denoted by $X$. Let $\C_{\bullet} $ an object of $\sCat$, then $\C_{\bullet +}\land X_{+}= (\C_{\bullet} \otimes X)_{+}$.
In order to show that $X\land - $ is a left Quillen functor, it is sufficient to show that the image of generating (acyclic) of $\sCat_{\ast}$ are (acyclic) cofibrations. We start with the case where $X$ has a disjoint base point. Consider the following pushout diagram:
  $$\xymatrix{
   \pi_{\bullet} Dec \Delta^{0}_{+}\land \pi_{\bullet} Dec A_{+} \ar[r] \ar[d] & \pi_{\bullet} Dec X_{+}\land \pi_{\bullet} Dec A_{+} \ar[d]\ar@/^/[rdd]  \\
    \pi_{\bullet} Dec \Delta^{0}_{+}\land \pi_{\bullet} Dec B_{+} \ar[r]  \ar@/_/[rrd]  & P  \ar@{.>}[rd] \\
     & &  \pi Dec X_{+}\land \pi_{\bullet} Dec B_{+}    
  }$$
  where $A\rightarrow B$ is an generating (acyclic) cofibration in $\sSet$. 
  The previous diagram is equivalent to the pushout diagram:
  $$\xymatrix{
   \pi_{\bullet} Dec (\Delta^{0}\times A)_{+} \ar@{^{(}->}[r] \ar@{^{(}->}[d] &\pi_{\bullet} Dec (X\times A)_{+} \ar@{^{(}->}[d]\ar@/^/[rdd]  \\
     \pi_{\bullet} Dec (\Delta^{0}\times B)_{+}\ar@{^{(}->}[r]  \ar@/_/[rrd]  & \mathbf{P}_{\bullet}  \ar@{^{(}.>}[rd] \\
     & &  \pi_{\bullet} Dec (X\times B)_{+}    
  }$$
 We have $\mathbf{P}_{\bullet} =\pi_{\bullet} Dec (\Delta^{0}\times B\sqcup_{\Delta\times A} X\times A)_{+}$ since $\pi_{\bullet}$ $Dec$ and $(-)_{\ast}$ commuteswith colimits, and the unique morphism  $\pi_{\bullet} Dec (\Delta^{0}\times B\sqcup_{\Delta^{0}\times A} X\times A)_{+}\rightarrow \pi Dec (X\times B)_{+}  $ is obviously an (acyclic) cofibration in $\sCat_{\ast}$ since $\sSet$ is monoidal model category, and $\pi_{\bullet}, ~(-)_{+},~Dec$ are left Quillen functors.
 Now, we have to show that $\pi_{\bullet} Dec (X\land A_{+})\rightarrow \pi_{\bullet} Dec (X\land B_{+})$ is a cofibration (acyclic cofibration) in $\sCat_{\ast} $ i.e., has a lifting property with respect to the acyclic fibrations  (resp. fibrations). 
  $$\xymatrix{
  \C_{\bullet} \ar@{->>}[r] & \D_{\bullet}
  }$$
  The following diagram summarize the situation:
   $$\xymatrix{
\pi_{\bullet} Dec (\Delta^{0}\times A)_{+} \ar@{^{(}->}[r]\ar@{^{(}->}[d] & \pi_{\bullet} Dec (X\times A)_{+} \ar@{^{(}->}[d]^-{g}   \ar[rr] &&   \pi_{\bullet} Dec (X\land A_{+} )  \ar[r] \ar[d] & \C \ar@{->>}[d]\\
 \pi_{\bullet} Dec (\Delta^{0}\times B)_{+}\ar@{^{(}->}[r] \ar@{.>}[rrrru]^-{(0)}&  \mathbf{P}_{\bullet} \ar@{^{(}->}[r]_-{f} \ar@{.>}[rrru]^-{(1)}& \pi_{\bullet} Dec (X\times B)_{+} \ar[r] \ar@{.>}[rru]^-{(2)} &  \pi_{\bullet} Dec (X\land B_{+})  \ar[r] \ar[r] \ar@{.>}[ru]^-{(3)} &\D_{\bullet} 
 }$$
The morphism $(0):  \pi_{\bullet} Dec (\Delta^{0}\times B)_{+}\rightarrow \C_{\bullet}$ is the obvious morphism which sends  everything to the base point of $\C$. The arrow $(1): \mathbf{P}_{\bullet}\rightarrow \C_{\bullet}$ is constructed by the universal property of the pushout
    $$\xymatrix{
   \pi_{\bullet} Dec \Delta^{0}_{+}\land \pi_{\bullet} Dec A_{+} \ar[r] \ar[d] & \pi_{\bullet} Dec X_{+}\land \pi Dec A_{+} \ar[d]\ar@/^/[rdd]  \\
    \pi_{\bullet} Dec \Delta^{0}_{+}\land \pi_{\bullet}Dec B_{+} \ar[r]  \ar@/_/[rrd]^{(0)}  & \mathbf{P}_{\bullet}  \ar@{.>}[rd]^{(1)} \\
     & & \C_{\bullet}   
  }$$
 Than, the arrow (2) is a lifting of the (acyclic) cofibration $f$. 
 
 Finally, we construct the third arrow  $(3): \pi_{\bullet} Dec (X\land B_{+}) \rightarrow \C_{\bullet}$ which makes the diagram commutes by the universal property of colimits. In fact,  the following diagram is a pushout in 
  $\sCat$.  
  
  $$\xymatrix{
  \pi_{\bullet}Dec(\Delta^{0}\times B)_{+}\ar[d]\ar[r] & \pi_{\bullet} Dec (X\times B)_{+}\ar[d]\ar@/^/[rdd]^{(2)}\\
  \ast\ar[r]\ar@/_/[rrd]^{(0)} & \pi_{\bullet}Dec(X\land B_{+}) \ar@{.>}[rd]^{(3)}\\
  && \C
  }$$
  
 %$$
 % \xymatrix{ 
%\pi_{\bullet} Dec (\Delta^{0}\times B)_{+} \ar[d] \ar[r]& \pi_{\bullet} Dec (X\times B)_{+} \ar[d] \ar@/^/[rdd]^{(2)}\\
  %  \ast \ar[r]\ar@/_/[rrd]^{(0)} &  \pi_{\bullet} Dec (X\land B_{+})  \ar@{.>}[rd]^{(3)}\\
   % & & \C  
   % }
%$$
 because the functor $\pi_{\bullet} Dec$ commutes with colimits. 
 
 we conclude that  $X_{\bullet}\land -$ is a left Quillen functor, consequently $(-)^{X_{\bullet}}$ is a right Quillen functor.
  \end{proof}

%sp(scat)
\section{Spectra $\Spn(\sCat_{\ast})$ and algebraic $\K$-theory}\label{section5}
This section is the outcome of this article. We define categories which look like \textbf{Waldhausen categories} and we will suggest a new definition of algebraic $\K$-theory for pointed simplicial categories. In what follow, $\mathcal{M}$  is a cofibrantly generated model category, cellular and left proper, equipped with a left Quillen endofunctor $\mathrm{T}:\mathcal{M}\rightarrow\mathcal{M}$ with a corresponding right Quillen adjoint $\mathrm{U}$.
\begin{definition}
Objects of  $\Spn(\mathcal{M},\mathrm{T})$ are sequences  $X=\{X_{0},X_{1},\dots X_{n},\dots\}$  of objects in $\mathcal{M}$, equipped with sequence of compatible structural morphisms $\sigma_{X}^{n}:\mathrm{T} X_{n}\rightarrow X_{n+1}$  for all 
$n\in\mathbb{N}.$
Morphisms in $\Spn(\mathcal{M},\mathrm{T})$ between $X=\{X_{0},X_{1},\dots X_{n},\dots\}$ and $Y=\{Y_{0},Y_{1},\dots Y_{n},\dots\}$ are degree wise morphisms in$\M$ which commutes with the structural morphisms
i.e., we have commutative diagrams for each natural number $n$:
$$
  \xymatrix{ 
  \mathrm{T}X_{n} \ar[r]^{\mathrm{T}f_{n}} \ar[d]^{\sigma_{X}} & \mathrm{T}Y_{n} \ar[d]^{\sigma_{Y}}\\
   X_{n+1} \ar[r]^{f_{n+1}} & Y_{n+1}.
   }
$$
\end{definition}
\begin{definition}
A $\mathrm{U}$-spectra in $\Spn(\mathcal{M},\mathrm{T})$ is a sequence $X=\{X_{0},X_{1},\dots X_{n},\dots\}$ such that  $X_{n}$ is fibrant in $\mathcal{M}$ for all $n$ and the adjoint map of $\sigma_{X}:\mathrm{T} X_{n}\rightarrow X_{n+1}$  i.e.,  $\tau_{X}:X_{n}\rightarrow \mathrm{U}X_{n+1}$ is a weak equivalence in $\M$ for all $n$. 
\end{definition}

\begin{theorem}\label{spectrification}

 There exists a stable model structure on the \textbf{category of spectra} $\Spn(\mathcal{M},\mathrm{T})$ where fibrant objects are $\mathrm{U}$-spectres.
\end{theorem}
\begin{proof}
See \cite{hovey2001} Theorem 3.4.
\end{proof}

In the model structure $\Spn(\mathcal{M},\mathrm{T})$, the left Quillen functor $\mathrm{T}: \mathcal{M}\rightarrow\mathcal{M}$ is extended to a left Quillen functor  $\mathrm{T}: \Spn(\mathcal{M},\mathrm{T})\rightarrow\Spn(\mathcal{M},\mathrm{T})$  which admits a right adjoint denotes by $s_{-}$.  such that  $(s_{-}X)_{n}=X_{n+1}$ for $n>0$. \\
The adjunction $(\mathrm{T}, s_{-})$ is a Quillen equivalence (cf \cite{hovey2001} Theorem 3.9). We should remark that the derived functor  $L\mathrm{T}$ became an invertible endofunctor in the homotopy category $\Ho\Spn(\mathcal{M},\mathrm{T})$.\\
We have also a Quillen adjunction between $\mathcal{M}$ and $\Spn(\mathcal{M},\mathrm{T})$ given by:
$$\xymatrix{
\mathcal{M} \ar@<1ex>[r]^-{\mathrm{T}^{\infty}} &\Spn(\mathcal{M},\mathrm{T}) \ar@<1ex>[l]^-{(-)_{0}}
}$$
where  $ \mathrm{T}^{\infty}(X)= \{X,\mathrm{T}X,\mathrm{T}\mathrm{T}X,\dots\} $ and $\sigma_{X}^{n}=id_{\mathrm{T}^{n+1}X}$. The functor $(-)_{0}$ associate to each spectra  
$X=\{X_{0},X_{1},\dots X_{n},\dots\}$ the object $X_{0}$.
\begin{theorem}\label{spectrificationscat}
There is a cofibrantly generated stable model category structure on $\Spn(\sCat_{\ast},\Sigma)$.
\end{theorem}
\begin{proof}
The category  $\sCat_{\ast}$ verify the hypothesis of  \ref{spectrification} (cellular, left proper, cofibrantly generated), and the functor $\Sigma= -\land S^{1}: \sCat_{\ast}\rightarrow \sCat_{\ast}$, where $S^{1}$ is a simplicial model for a circle, is a left Quillen functor \ref{S-Quillen}, with a right adjoint denoted by $\Omega.$
We conclude the stable model structure on $\Spn(\sCat_{\ast},\Sigma)$ exists. 
\end{proof}
\begin{definition}\label{catwald}
A simplicial category is called a \textbf{weak complete Wladhausen category} if it is equivalent  to a $0$-object of some $\Omega$-spectre in the stable category of spectra $\Spn(\sCat_{\ast},\Sigma)$.
\end{definition}\label{faible}
In some sense, a weak Wladhausen category is an infinite loop space in the category of spectra
 $\sCat_{\ast}$.In order to justify this definition we compute the mapping space $\map_{\ast}$ of the model category  $\Spn(\sCat_{\ast},\Sigma).$
The following equivalences are a direct consequence of \ref{adjmap}.
 \begin{theorem}\label{adjmap}[\cite{DH2010}, Theorem 2.12.]
   Let the following Quillen adjunction between two model categories:

   $$\xymatrix{
\mathcal{C} \ar@<1ex>[r]^-{G} & \mathcal{M}. \ar@<1ex>[l]^-{F} 
}$$
then we have a natural ismorphism
$$ \map_{\mathcal{C}}(a,\mathrm{R}Fb)\rightarrow\map_{\mathcal{M}}(\mathrm{L}Ga,b)$$
in the category $\Ho(\sSet)$
\end{theorem}

 A consequence of theorem \ref{adjmap} in the case of  $\sCat$ where we consider the Quillen adjunction between $\sSet,~\sSet^{2}$ and $\sCat$ gives us the following result:
 \begin{corollary}
 Let $\C_{\bullet}$ be a fibrant object in $\sCat$ and $X\in\sSet.$ We have a natural ismorphism $\Ho(\sSet)$:
 $$ \map_{\sCat}(\pi d_{\ast}X,\C_{\bullet})\rightarrow \Map(X,\diag \N\iso \C_{\bullet}),$$
 where $\Map$ is the right adjoint to the cartesian product in $\sSet$.
 \end{corollary}
 \begin{proof}
The isomorphism $ \map_{\sCat}(\pi d_{\ast}X,\C_{\bullet})\rightarrow \map_{\sSet}(X,\diag \N\iso \C_{\bullet})$ is a direct consequence of \ref{adjmap}.
Since $\sSet$ is a simplicial model category we have that   
 $$\map_{\sSet}(X,\diag \N\iso \C_{\bullet})\simeq \Map(X,\diag \N\iso \C_{\bullet})$$
in $\Ho(\sSet)$.   
 
 \end{proof}

 There is a natural transformation between $\diag$ and the functor $\overline{W}$ which is a weak equivalence i.e., $\diag(X)\rightarrow \overline{W}(X)$ is a weak equivalence in $\sSet$, for any simplicial set $X$. \\
Let $\D_{\bullet}^{\bullet}=\{\D^{0}_{\bullet},\D^{1}_{\bullet},\dots \D^{n}_{\bullet},\dots\}$ be an $\Omega$-spectra in $\Spn(\sCat_{\ast},\Sigma)$. We have a the following corollaries:
\begin{corollary} 
Let $\C_{\bullet}$ be a simplicial (pointed) category in $\sCat$ equipped with the $\overline{W}$-model structure. 
The adjunction
$$\xymatrix{
\sSet\ar@<1ex>[r]^-{\pi~d_{\ast}} &\sCat\ar@<1ex>[l]^-{\diag\N\iso}
}$$
gives us the isomorphism  $\map_{\sCat}(\ast,\C_{\bullet})\sim \diag\N\iso\C_{\bullet}$ in $\Ho(\sSet).$
\end{corollary}

\begin{corollary} 
 If we denote by $S^{0}$ the constant simplicial category $\ast\sqcup\ast$, then the adjunction:
$$\xymatrix{
\sCat\ar@<1ex>[r]^-{(-)_{+}} &\sCat_{\ast}\ar@<1ex>[l]^-{F}
}$$
gives us the isomorphism  $\map_{\sCat_{\ast}}(S^{0},\C_{\bullet})\simeq\map_{\sCat}(\ast,\C_{\bullet})\simeq \diag\N\iso\C_{\bullet}$ in $\Ho(\sSet).$
\end{corollary}
\begin{corollary}  
The adjunction
$$\xymatrix{
\sCat_{\ast} \ar@<1ex>[r]^-{\Sigma^{\infty}} &\Spn(\sCat_{\ast},\Sigma) \ar@<1ex>[l]^-{(-)_{0}}
}$$
gives us the isomorphism $\map_{\Spn(\sCat_{\ast})}(\Sigma^{\infty}S^{0},\D_{\bullet}^{\bullet})\simeq\map_{\sCat_{\ast}}(S^{0},\D^{0}_{\bullet})\simeq \diag\N\iso\D^{0}_{\bullet}.$
\end{corollary} 
\begin{corollary} 
The adjunction 
$$\xymatrix{
\Spn(\sCat_{\ast},\Sigma) \ar@<1ex>[r]^-{\Sigma} &\Spn(\sCat_{\ast},\Sigma) \ar@<1ex>[l]^-{s_{-}}
}$$
induces an isomorphism
 $$\map_{\Spn(\sCat_{\ast})}(\Sigma\Sigma^{\infty}S^{0},\D_{\bullet}^{\bullet})\simeq\map_{\Spn(\sCat_{\ast})}(\Sigma^{\infty}S^{0},s_{-}\D_{\bullet}^{\bullet})\simeq\map_{\sCat_{\ast}}(S^{0},\D^{1}_{\bullet})\simeq\diag\N\iso\D^{1}_{\bullet}$$
and more generally 
 $$   \map_{\Spn(\sCat_{\ast})}(\Sigma^{n}\Sigma^{\infty}S^{0},\D_{\bullet}^{\bullet})\simeq\map_{\sCat_{\ast}}(S^{0},\D^{n}_{\bullet})\simeq\diag\N\iso\D^{n}_{\bullet} .$$
 \end{corollary} 
 \begin{remark}
 Let $S^{n}$ be a simplicial model for the sphere of dimension $n$, then $\pi_{\bullet} d_{\ast}S^{n}$ is a simplicial category and $\Sigma(\pi_{\bullet} d_{\ast}S^{n})\simeq \pi_{\bullet} d_{\ast}S^{n+1}$.
 \end{remark}
By definition of  $\map_{\sCat_{\ast}}$ and the fact that  $\D^{n}_{\bullet}\rightarrow \Omega\D^{n+1}_{\bullet}$ is an equivalence in $\sCat_{\ast}$ between fibrant objects, we deduce the following corollary. 
 \begin{corollary} \label{corwaldhausen}
 Using the precedent Quillen adjunctions and \ref{adjmap}, we have the following isomorphisms in $\Ho(\sSet):$
 \begin{enumerate}

 \item $\map_{\sCat_{\ast}}(\Sigma S^{0},\D^{n+1}_{\bullet})\simeq\map_{\sCat_{\ast}}(S^{0},\Omega\D^{n+1}_{\bullet})\simeq\diag\N\iso\Omega\D^{n+1}_{\bullet}.$

\item  $\map_{\sCat_{\ast}}(\Sigma S^{0},\D^{n+1}_{\bullet})\simeq\map_{\sCat_{\ast}}(\pi_{\bullet}d_{\ast}S^{1},\D^{n+1}_{\bullet})\simeq \Omega \diag\N\iso\D^{n+1}_{\bullet}.$

\item  $\map_{\sCat_{\ast}}(\Sigma S^{0},\D^{n+1}_{\bullet})\simeq \map_{\sCat_{\ast}}(S^0,\Omega\D^{n+1}_{\bullet})\simeq\map_{\sCat_{\ast}}(S^0,\D^{n}_{\bullet})\simeq\diag\N\iso\D^{n}_{\bullet}.$\\\\
 \end{enumerate}

 \end{corollary}

%\begin{corollary} 
% $$\Omega\map_{\ast}(S^0,\D^{n+1}_{\bullet})\simeq\Omega\diag\N\iso\D^{n+1}_{\bullet}\simeq\diag\N\iso\D^{n}_{\bullet}\simeq\diag\N\iso\Omega\D^{n+1}_{\bullet}.$$
 %\simeq \map_{\ast}(S^0,\Omega\D^{n+1}_{\bullet}).$$
%\end{corollary} 

\subsection{Algebraic $\K$-theory} As before, we suppose that  $\D_{\bullet}^{\bullet}=\{\D^{0}_{\bullet},\D^{1}_{\bullet},\dots \D^{n}_{\bullet},\dots\}$ is an $\Omega$-spectra in $\Spn(\sCat_{\ast},\Sigma).$
In general the sequence of simplicial sets 
$$\{\map_{\sCat_{\ast}}(S^{0}, \D^{0}_{\bullet}),\map_{\sCat_{\ast}}(S^{0}, \D^{1}_{\bullet}),\dots\}$$ 
does not form a spactra in $\Spn(\sSet_{\ast},\Sigma).$
This sequence is not an element of $\Spn(\sSet_{\ast},\Sigma)$ but it has the property of an $\Omega$-spectra, i.e., $$ \map_{\sCat_{\ast}}(S^{0}, \D^{n}_{\bullet})\simeq \Omega\map_{\sCat_{\ast}}(S^{0}, \D^{n+1}_{\bullet}),~\forall n\in\mathbb{N}.$$
In some sense, an $\Omega$-spectra $\D^{\bullet}_{\bullet}$ has the property that $\D^{n+1}_{\bullet}$ is a model for the Waldhausen $\mathcal{S}_{\bullet}$-construction for $\D^{n}_{\bullet}$, i.e., 
$\D^{n+1}_{\bullet}$ is a model for $\mathcal{S}_{\bullet}\D^{n}_{\bullet}$. Or equivalently, 
$\D^{n+1}_{\bullet}$ is a \textit{categorical delooping} for $\D^{n}_{\bullet}$.
%$\pi_{i}\map(S^{0}, \D^{n}_{\bullet})$ est isomorphe à $\pi_{i+1}\map(S^{0}, \D^{n+1}_{\bullet})$ pour tout $i$ et tout $n$.
 \begin{definition}
 Let $\C_{\bullet}$ be a simplicial category (i.e., an objet of $\sCat_{\ast}$) which is a weak complete Waldhausen category \ref{catwald}, we define the algebraic $\K$-theory of $\C_{\bullet}$
 by the simplicial set $\map_{\sCat_{\ast}}(S^{0},\C_{\bullet})$, so $\K_{i}(\C_{\bullet})=\pi_{i}\map_{\sCat_{\ast}}(S^{0},\C_{\bullet}).$
 \end{definition}
 
%%%%%%%%%%%%%%%%%%%%%%%%%%%%%%%%%%%%%
%%%%%%%%%%%%%%%%%%%%%%%%%%%%%%%%%%%%%

\appendix\label{A}
\section{}
\begin{lemma}\label{petitscat0}
Let
$$\xymatrix{
\mathcal{C} \ar@<1ex>[r]^-{G} & \mathcal{D}\ar@<1ex>[l]^-{F}
}$$
be an adjunction, such that $F$ commutes with directed  colimits. If $C\in\mathcal{C}$ is a small object for a certain ordinal $\beta$, then $G(C)$ is small in $\mathcal{D}$.
\end{lemma}
\begin{proof}
Suppose that we have a directed  colimit $\colim_{\alpha<\beta}T_{\alpha}$ in $\mathcal{D}$. We have the following sequence of isomorphisms
  \begin{eqnarray*}
  \homs_{\mathcal{D}}(G(C),\colim_{\alpha<\beta}T_{\alpha})&\simeq&\homs_{\mathcal{C}}(C,~F\colim_{\alpha<\beta}T_{\alpha})  \\
&\simeq& \homs_{\mathcal{C}}(C,\colim_{\alpha<\beta} ~F(T_{\alpha}))  \\
&\simeq& \colim_{\alpha<\beta}\homs_{\mathcal{C}}(C,~F(T_{\alpha})) \\
&\simeq&  \colim_{\alpha<\beta}\homs_{\mathcal{D}}(G(C),T_{\alpha})
\end{eqnarray*}
The sequence ismomorphism, is a consequence of the fact that  $F$ commutss with directed  colimits. The rest of isomorphisms are  obvious because  $C$ is $\beta$-small by definition.
\end{proof}

\begin{lemma}
In Moerdijk's model category on bisimplicial sets, domains and codomains of  generating (acyclic) cofibrations   $\mathrm{I}$ ($\mathrm{J}$) are small.
\end{lemma}
\begin{proof}
The generating (acyclic) cofibration in $\sSet^{2}$ are the image of generating (acyclic) cofibration via the fonctor $d_{\ast}$ of generating (acyclic) cofibration $\sSet$.
$$\xymatrix{
\sSet \ar@<1ex>[r]^-{d_{\ast}} & \sSet^{2}\ar@<1ex>[l]^-{\diag}
}$$
We recall that  $\diag$ admits also a right adjoint denoted by $d^{\ast}$, so $\diag$ commutes with colimits. Moreover objects in $\sSet$ are small for a certain ordinal.
Let $A$ be a (co)domain of a generating acyclic cofibration in $\sSet$, then  $d_{\ast}A$ is small in $\sSet^{2}$ by the lemma \ref{petitscat0}.
\end{proof}
\begin{lemma}\label{smallscat}
The functor $\diag\N\iso: \sCat\rightarrow \sSet$ commutes with directed  colimits.
\end{lemma}
\begin{proof}
Let  $\colim_{\lambda}\C^{\lambda}_{\bullet}$ be a directed  colimit in $\sCat$, for a certain ordinal $\lambda$.
\begin{eqnarray*}
\Big(\diag\N\iso (\colim_{\lambda}\C^{\lambda}_{\bullet})\Big)_{n}& = & \homs_{\sSet}(\Delta^{n}, \diag\N\iso (\colim_{\lambda}\C^{\lambda}_{\bullet})) \\ 
 & = & \homs_{\sCat}(\pi_{\bullet}~d_{\ast}(\Delta^{n}),\colim_{\lambda}\C^{\lambda}_{\bullet}) \\ 
  & = & \homs_{\sCat}(\pi_{\bullet}(\sqcup_{\Delta^{n}}\Delta^{n}),\colim_{\lambda}\C^{\lambda}_{\bullet}) \\ 
   & = & \homs_{\sCat}(\sqcup_{\Delta^{n}}\pi(\Delta^{n}),\colim_{\lambda}\C^{\lambda}_{\bullet}) \\ 
&=&\homs_{\Cat}(\pi\Delta^{n},\colim_{\lambda}\C^{\lambda}_{n})\\
&=&\colim_{\lambda}\homs_{\Cat}(\pi\Delta^{n},\C^{\lambda}_{n})\\
&=&\colim_{\lambda}\homs_{\sCat}(\pi_{\bullet}~ d_{\ast}\Delta^{n},\C^{\lambda}_{\bullet})\\
&=&\colim_{\lambda}\homs_{\sSet}(\Delta^{n},\diag\N\iso\C^{\lambda}_{\bullet})\\
&=& \colim_{\lambda}\Big( \diag\N\iso\C^{\lambda}_{\bullet}\Big)_{n}.
\end{eqnarray*}

All the isomorphisms are consequence of adjunctions. The fifth isomorphisms is due to the fact that $\pi\Delta^{n}$ is a small object in $\Cat$. 
\end{proof}

\begin{lemma}\label{petitscat}
The domains and codomains of generating (acyclic) cofibrations in $\sCat$ are small.
\end{lemma}
\begin{proof}
It is a consequence of  \ref{petitscat0}, \ref{smallscat} and the fact that all objects in $\sSet$ are small for a cetrain ordinal.
\end{proof}

%%%%%%%%%%%%%%%%%%%%%%%%%%%%%%%%%%
%%%%%         Path object                                        %%%%%%%%%%
%%%%%%%%%%%%%%%%%%%%%%%%%%%%%%%%%%

%%%%%%%%%%%%%%%%%%%%%%%%%%%%%%%%%%
%%%%%         proof Path object                              %%%%%%%%%%
%%%%%%%%%%%%%%%%%%%%%%%%%%%%%%%%%%

%------------------------------------------------------------------------------
% Part 3
%------------------------------------------------------------------------------

%------------------------------------------------------------------------------
% Part 4
%------------------------------------------------------------------------------

%------------------------------------------------------------------------------
% Bibliothèque 
%------------------------------------------------------------------------------
\bibliographystyle{plain} 
\bibliography{speccat}

\end{document}